\documentclass[12pt]{amsart}

\usepackage{amsmath, amsthm, amssymb}

\usepackage[dvipsnames]{xcolor}
\definecolor{darkgreen}{RGB}{0,150,0}
\definecolor{darkred}{RGB}{200,0,0}
\definecolor{darkblue}{RGB}{0,0,220}
\definecolor{gold}{RGB}{255,218,9}
\usepackage[colorlinks, linkcolor=darkblue, citecolor=darkgreen, urlcolor=darkgreen]{hyperref}

\usepackage{newpxtext}
\usepackage[bigdelims,vvarbb]{newpxmath}    
\usepackage[T1]{fontenc}  

\usepackage{setspace}
\usepackage[margin=1in]{geometry}
             
\usepackage{comment}
\usepackage{thmtools, thm-restate}
\usepackage{overpic}
\usepackage{subcaption}

\usepackage{aliascnt}
\usepackage[noabbrev,capitalize,nameinlink]{cleveref}

\newtheorem{theorem}{Theorem}[section]

\newtheorem{lemma}[theorem]{Lemma}
\newtheorem{corollary}[theorem]{Corollary}

\theoremstyle{definition}

\newtheorem{question}[theorem]{Question}
\theoremstyle{remark}
\newtheorem{remark}[theorem]{Remark}
\numberwithin{equation}{section}

\usepackage{mleftright}

\newcommand{\C}{\mathbb{C}}

\newcommand{\ve}{\varepsilon}

\makeatletter 
 \def\l@subsection{\@tocline{2}{0pt}{2pc}{6pc}{}} \makeatother

\begin{document}

\title{Lagrangian slice disks with symplectomorphic exteriors}

\author{Joseph Breen}
\address{University of Alabama, Tuscaloosa, AL 35401}
\email{jjbreen@ua.edu} \urladdr{https://sites.google.com/view/joseph-breen}


\thanks{JB was partially supported by an AMS-Simons Travel Grant.}

\begin{abstract}
By modifying a construction of Abe and Tange, we exhibit arbitrarily large families of Lagrangian slice disks with Weinstein deformation equivalent exteriors. This answers a Lagrangian version of a question of Hitt and Sumners. We raise other open questions related to Lagrangian slice disks and their exteriors. 
\end{abstract}

\maketitle

\tableofcontents

\section{Introduction}

Let $D\subset B^4$ be a slice disk and $X_D:= B^4 \setminus \nu(D)$ its exterior. In 1981, Hitt and Sumners \cite{hitt1981many} studied the \emph{indeterminacy index} $\zeta(D)$, defined as the number of non-isotopic slice disks in $B^4$ whose exteriors are diffeomorphic to $X_D$. (Here, isotopy is not relative to the boundary.) In the context of a growing body of literature on the knot complement problem \cite{gluck1961embedding,lashof1969codimensiontwo,cappell1976complement,gordon1976knots}, their work demonstrated that the disk complement problem is more complicated. After establishing bounds in various settings, they asked if $\zeta$ could be arbitrarily large:

\begin{question}\cite[Question 1]{hitt1981many}\label{question:HS81}
Fix $N\geq 1$. Does there exist $D\subset B^4$ with $\zeta(D) \geq N$?     
\end{question}

\noindent The work of Hitt and Sumners and its immediate progeny \cite{plotnick1983infinitely,suciu1985infinitely} primarily concerned codimension-$2$ knotting in higher dimensions. There remained little progress on the (as stated) low-dimensional \cref{question:HS81} for many years. 

A stronger question \cite[Question 2]{hitt1981many} asking whether there exists a slice disk with $\zeta(D) = \infty$ was eventually answered affirmatively by Abe and Tange \cite{abe2022ribbon}, who found an infinite family of non-isotopic ribbon disks, distinguished by the knot types of their boundaries, with diffeomorphic exteriors. Later, Meier and Zupan \cite{meier2023knots}, using a prior construction in \cite{meier2022generalized}, strengthened this result by exhibiting infinitely many knots such that each bounds an infinite family of non-isotopic ribbon disks. 

There is a rich tradition of importing low-dimensional surgery-theoretic questions of interest to the Legendrian setting, where they typically remain interesting. For example, the knot complement problem \cite{gordon1989complement} births its Legendrian counterpart \cite{kegel2018legendrian}, while surgery numbers \cite{auckly1993surgery}, characterizing slopes \cite{kronheimer2007monopoles,lackenby2019every,piccirillo2019snake}, and cosmetic surgery \cite{gordon1991dehn} each have burgeoning contact interpretations in \cite{ding2009handle,etnyre2023contactsurgerynumbers}, \cite{casals2024steintrace,kegel2025share}, and \cite{etnyre2025cosmetic}. Given this, it is natural to cast the disk exterior problem of Hitt and Sumners in an appropriate symplectic setting.

\subsection{Main results} Endow $B^4$ with its standard symplectic structure $\omega_{\mathrm{st}}$. A Lagrangian slice disk $D\subset B^4$ with Legendrian boundary is \emph{regular} \cite{eliashberg2018flexiblelagrangians} if there is a Weinstein homotopy $(\lambda_t, \phi_t)$ of Weinstein structures on $B^4$, $0\leq t\leq 1$, such that $(\lambda_0,\phi_0)$ is the standard radial structure on $B^4$, $D$ is Lagrangian with respect to $d\lambda_t$ for all $t$, and $D$ is a co-core of the Weinstein structure $(\lambda_1,\phi_1)$. In particular, if $D$ is regular then its exterior naturally inherits the structure of a Weinstein, hence Stein, domain. Note that \emph{decomposable} disks are regular \cite{conway2021symplectic}, but we do not know if the converse holds. 

Call the \emph{Weinstein indeterminacy index}, denoted $\zeta_{\omega}(D)$, the number of distinct (up to Hamilto\-nian iso\-topy, not relative to the boundary) regular Lagrangian slice disks whose exteriors are Weinstein deformation equivalent to $X_D$. (Weinstein deformation equivalence of domains implies exact symplectomorphism of their unique completions. For this reason, we will informally call the exteriors symplectomorphic.) Here we propose a Lagrangian version of \cref{question:HS81}.

\begin{question}\label{question:main}
Fix $N\geq 1$. Does there exist a regular Lagrangian slice disk $D\subset (B^4, \omega_{\mathrm{st}})$ with $\zeta_{\omega}(D) \geq N$?
\end{question}

We emphasize that \cref{question:main} is not relative to the boundary. Prior work in the Lagrangian setting has focused on distinction up to Hamiltonian isotopy relative boundary for a fixed knot type; for instance, see the slice disks of Li and Tange \cite{li2021smoothly}, the discussion after \cref{question:same-knot}, and more generally the study of exact Lagrangian fillings \cite{ekholm2012exactcobordisms,casals2022infinitely}. Typically, such surfaces become (Lagrangian) isotopic when considered not relative to the boundary. 

Our main result involves modifying the construction of Abe and Tange in \cite{abe2022ribbon} and upgrading it to the Weinstein setting to answer \cref{question:main} affirmatively. 

\begin{theorem}\label{thm:main}
Fix $N\geq 1$. The Legendrian knots $\Lambda_{n,N}$, $1\leq n \leq 2N$, in \cref{fig:main} satisfy: 
\begin{enumerate}
    \item Each $\Lambda_{n,N}$ bounds a regular (in fact, decomposable) Lagrangian slice disk $D_{n,N}$ such that the exteriors $X_{D_{n,N}}$, $1\leq n \leq 2N$, are all Weinstein deformation equivalent.\label{part:main-construction}
    \item For $N$ sufficiently large, the knots $\Lambda_{1,N}, \dots, \Lambda_{N,N}$ are pairwise smoothly non-isotopic. In particular, the disks $D_{1,N}, \dots, D_{N,N}$ are pairwise smoothly non-isotopic.\label{part:main-obstruction}

\end{enumerate}
Consequently, $\zeta_{\omega}(D_{1,N})\geq N$. 
\end{theorem}

\begin{figure}[ht]
	\centering
    \begin{overpic}[width=0.85\textwidth]{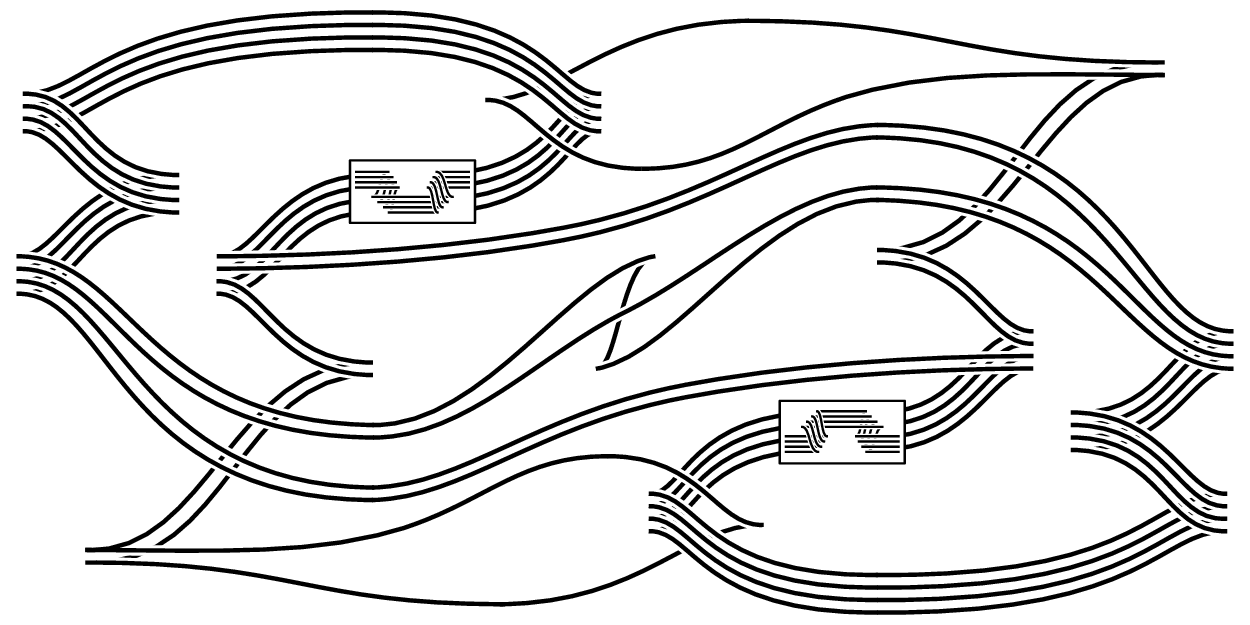}
       \put(2,13){\small $\Lambda_{n,N}$}
       \put(62.5,10.5){\footnotesize $n$ copies}
       \put(25.5,38.25){\footnotesize $2N-n$ copies}
	\end{overpic}
	\caption{The knots $\Lambda_{n,N}$ in the statement of \cref{thm:main}.}
	\label{fig:main}
\end{figure}

Excising a neighborhood of a Lagrangian slice disk $D$ from $(B^4, \omega_{\mathrm{st}})$ induces contact-$(+1)$ surgery on the boundary. Etnyre \cite{etnyre2008contact} constructed the first pairs of distinct Legendrian knots sharing a contact-$(+1)$ surgery, and other interesting pairs were recently found by Kegel and Piccirillo \cite{kegel2025share}. Casals, Etnyre, and Kegel \cite[Corollary 4.6]{casals2024steintrace} produced an infinite family of distinct Legendrian knots sharing contact-$(+1)$ surgeries. Their surgeries are overtwisted. In contrast, an immediate corollary of \cref{thm:main} is the existence of arbitrarily large collections of distinct Legendrian knots sharing tight contact-$(+1)$ surgeries.

\begin{corollary}
For any $N\geq 1$, there is a collection of $N$ pairwise distinct Legendrian knots sharing a common Stein fillable (hence tight) contact-$(+1)$ surgery. 
\end{corollary}

\begin{remark}[Comparing with \cite{abe2022ribbon,meier2023knots}]\label{remarK:comparison}
Our construction is distinct from \cite{abe2022ribbon,meier2023knots}. In fact, the families considered in the cited works definitively cannot be used to answer \cref{question:main} or its implied infinite version (see \cref{question:infinity}). For one, the seed knot in the family considered by Abe and Tange in \cite{abe2022ribbon} is $4_1 \, \# \, \overline{4_1}$, which is not Lagrangian slice \cite{cornwell2016obstructions}. Moreover, their sequential construction produces handlebody decompositions of the exteriors using $2$-handles with increasingly large framing. Given the framing restrictions on Stein handlebodies \cite{gompf1998handlebody}, it is not possible for all of these to admit Stein structures.

Likewise, none of the generalized square knots $Q_{p,q}:=T_{p,q} \, \#\, \overline{T_{p,q}}$ for $p > q > 1$ considered by Meier and Zupan in \cite{meier2022generalized,meier2023knots} are Lagrangian slice. Indeed, combining the connected sum formula for the maximum Thurston-Bennequin invariant $\mathrm{TB}$ due to \cite{etnyre2003connected,torisu2003additivity} and the torus link classification of \cite{etnyre2001knotscontact},
\begin{align*}
    \mathrm{TB}(T_{p,q} \, \#\, \overline{T_{p,q}}) &= \mathrm{TB}(T_{p,q}) + \mathrm{TB}(T_{-p,q}) + 1 \\
    &= (pq - p - q) + (-pq) + 1 \\
    &= - p - q + 1.
\end{align*}
Therefore, $\mathrm{TB}(Q_{p,q}) < -1$. However, any Lagrangian slice knot type $K$ has $\mathrm{TB}(K) = -1$ by Chantraine \cite{chantraine2010concordance}.
\end{remark}

Prompted by \cref{remarK:comparison}, it is natural to ask whether $K\, \#\, \overline{K}$ is ever Lagrangian slice (see \cref{question:KplusKbar} below). In \cref{sec:connected-summation} we will prove the following result which may be of independent interest. 

\begin{theorem}\label{thm:KplusKbar}
If $K\, \#\,  \overline{K}$ is Lagrangian slice, then $K$ has trivial HOMFLYPT polynomial.     
\end{theorem}

\noindent Whether or not there exists a nontrivial knot with trivial HOMFLYPT polynomial is a well-known open question.

\begin{remark}[Annulus twisting]
Another feature of the knots constructed by Abe and Tange is that they are obtained from Osoinach's annulus twist operation \cite{osoinach2006manifolds,abe2013annulus}. Casals, Etnyre, and Kegel \cite{casals2024steintrace} generalized the annulus twist and annulus presentation to the contact setting, but our knots do not arise from this construction. Their contact annulus presentation, which is particularly suited for constructing interesting Stein traces (hence contact-$(-1)$ surgeries), yields a knot with $\mathrm{tb} = 1.$ (They give a more general construction that can produce $\mathrm{tb}=-1$ knots, but these are in general stabilized, hence non-fillable.)  
\end{remark}

We distinguish our knots in \cref{sec:obstruction} using a hyperbolic volume argument along the lines of \cite{osoinach2006manifolds}. In contrast, Abe and Tange's knots were distinguished in \cite{takioka2019classification} through an involved calculation of the $0^{\mathrm{th}}$ coefficient polynomial of the HOMFLYPT polynomial. It is plausible that this method, together with the twist family formulas established by Kegel and Piccirillo in \cite{kegel2025share}, could strengthen the statement of \cref{thm:main} and distinguish all $\Lambda_{n,N}$ without assuming $N$ sufficiently large. We instead opt for hyperbolic brevity.

\subsection{Open questions}

We close with some open questions prompted by our work. The first is the natural Lagrangian analogue of \cite[Question 2]{hitt1981many} and its surgery-theoretic counterpart: 

\begin{question}\label{question:infinity}
Does there exist a regular Lagrangian slice disk $D$ with $\zeta_{\omega}(D) =\infty$? More generally, does there exist an infinite family of distinct Legendrian knots with contactomorphic tight contact-$(+1)$ surgeries?
\end{question}

Next, recall that Meier and Zupan exhibited ribbon disks which are non-isotopic, have diffeomorphic exteriors, and yet fill the same knot. 

\begin{question}\label{question:same-knot}
Does there exist a Legendrian knot $\Lambda$ bounding two non-isotopic (Ham\-iltonian or smooth, not relative to the boundary) regular Lagrangian slice disks $D,D'$ such that $X_D$ and $X_{D'}$ are Weinstein deformation equivalent?    
\end{question}

\noindent It is known that many Legendrian knots, including the simplest nontrivial Lagrangian slice knot $\overline{9_{46}}$, bound multiple Lagrangian disks which are distinct up to Hamiltonian (in fact, smooth) isotopy relative boundary \cite{ekholm2016nonloose,li2021smoothly}. However, to the best of the author's knowledge, such disk pairs in the literature are either Hamiltonian isotopic, or if not, they have non-homeomorphic exteriors. (There is a contactomorphism involution on $(S^3, \xi_{\mathrm{st}})$ identifying the two Reeb chords producing the distinct slice disks for $\overline{9_{46}}$ considered in \cite{ekholm2016nonloose,li2021smoothly}. The contact isotopy induced by the involution can be extended into the symplectization as a Hamiltonian isotopy, not relative boundary.) In a symplectically dual vein, Hayden \cite{hayden2021exotically} produced infinitely many knots bounding distinct holomorphic, hence symplectic, slice disks, but again these have non-homeomorphic exteriors.    

Moving on, the knots constructed in \cref{thm:main} have Weinstein handlebody exteriors such that at least one of the $2$-handle attaching spheres is Legendrian destabilizable; see \cref{fig:main-handle-diagram}. This characteristic is crucial for the efficacy of the construction. 

\begin{question}\label{question:loose}
Does there exist a regular Lagrangian slice disk $D$ with $\zeta_{\omega}(D) > 1$ such that the $2$-handles comprising the handle decomposition of $X_D$ are non-de\-stabilizable?     
\end{question}

\noindent In higher dimensions, a regular Lagrangian disk is \emph{flexible} if its exterior is loose \cite{eliashberg2018flexiblelagrangians}. \cref{question:loose} may be thought of as asking for a lower-dimensional analogue of a ``non-flexible'' disk. 

As mentioned above, \cref{thm:KplusKbar} prompts the following simple question which does not appear to be explicitly asked elsewhere: 

\begin{question}\label{question:KplusKbar}
Is there a nontrivial Lagrangian slice knot type of the form $K \,\#\,\overline{K}$? 
\end{question}

\subsection*{Organization} In \cref{sec:construction}, we modify the construction of Abe and Tange to produce Lagrangian disks with equivalent exteriors, proving \eqref{part:main-construction} of \cref{thm:main}. In \cref{sec:obstruction} we use hyperbolicity to distinguish sufficiently many knots, proving \eqref{part:main-obstruction} of \cref{thm:main}. Finally, in \cref{sec:connected-summation} we prove \cref{thm:KplusKbar} on Lagrangian sliceness of $K \, \#\, \overline{K}$. 

\subsection*{Acknowledgments} The author would like to thank Sean Eli, B\"ulent Tosun, and Alex Zupan for their interest in our work, and John Etnyre for explaining \cref{lemma:sum}. We additionally thank an anonymous referee for identifying errors and making suggestions which improved the clarity of the article.

\section{Construction}\label{sec:construction}

In this section we prove \eqref{part:main-construction} of \cref{thm:main} by constructing the family of Lagrangian slice disks with Weinstein deformation equivalent exteriors. The proof relies on the following lemma, which is our Legendrian replacement for the mechanism in \cite[Figure 4]{abe2022ribbon}. Informally, it allows for the transfer of a $\pm$-double Legendrian stabilization from one strand to another by handlesliding over a contact-$(+1)$ ``meridional'' surgery.

\begin{lemma}\label{lemma:main}
Let $U$ be the max-tb Legendrian unknot in a Darboux ball. After performing contact-$(+1)$ surgery along $U$, the two-component Legendrian tangles depicted in \cref{fig:main-lemma} are Legendrian isotopic relative to their endpoints.     
\end{lemma}

\begin{figure}[ht]
	\centering
    \begin{overpic}[width=0.8\textwidth]{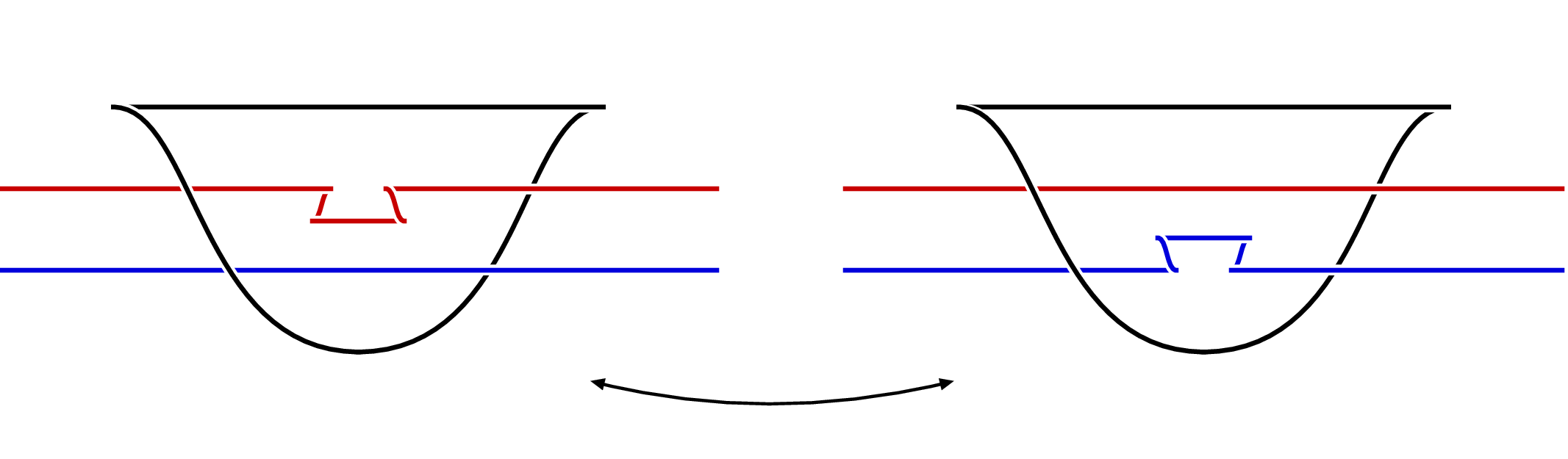}
       \put(20,24.25){\small $(+1)$}
       \put(4,22){\small $U$}
       \put(74,24.25){\small $(+1)$}
       \put(58,22){\small $U$}
	\end{overpic}
	\caption{The statement of \cref{lemma:main}.}
	\label{fig:main-lemma}
\end{figure}

In the proof of the lemma, and throughout the rest of the article, we make frequent use of handleslides across contact-$(\pm 1)$ surgeries. Models for contact handleslides in dimension $3$ were first developed by Ding and Geiges \cite{ding2009handle}; we will appeal to models developed by Casals and Murphy \cite[\S 2]{casals2019affine} in all dimensions. For review, the relevant models in dimension $3$ are presented in \cref{fig:slides}.

\begin{figure}[ht]
	\centering
	\begin{overpic}[scale=.45]{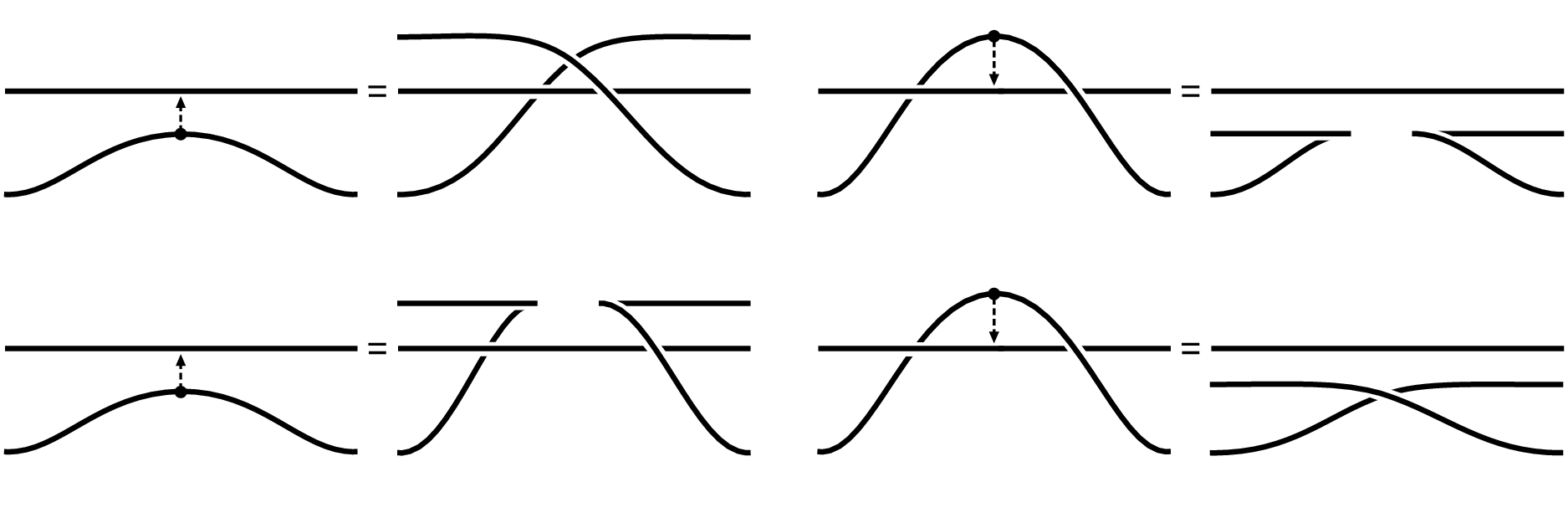}
        \put(18,28.5){\tiny $(+1)$}
        \put(43.25,28.5){\tiny $(+1)$}
        \put(70,28.5){\tiny $(+1)$}
        \put(95,28.5){\tiny $(+1)$}

        \put(18,12){\tiny $(-1)$}
        \put(43.25,12){\tiny $(-1)$}
        \put(70,12){\tiny $(-1)$}
        \put(95,12){\tiny $(-1)$}
       
	\end{overpic}
	\caption{Front projection models for contact-$(\pm 1)$ handleslides, as in \cite[Proposition 2.17]{casals2019affine}.}
	\label{fig:slides}
\end{figure}

\begin{proof}[Proof of \cref{lemma:main}.]
The proof is contained in \cref{fig:main-lemma-proof-new} through a sequence of handleslides and Legendrian Reidemeister moves. We refer to the $12$ frames in the order indicated by the arrows. First, from frame $1$ to frame $2$ we perform the indicated contact-$(+1)$ handleslide. Frame $3$ follows after Reidemeister II moves which poke and unpoke the red zig-zags across the black strands. Frame $4$ follows from a Reidemeister I, and frame $5$ is obtained by a sequence of Reidemeister III moves transferring the black crossing past the red and blue strands. In frame $6$ we perform a simplifying Reidemeister I move on the black, and re-poke the red zig-zags across the black with Reidemeister II moves. Additional Reidemeister II moves are performed to produce frame $7$, preparing an additional contact-$(+1)$ handleslide which yields frame $8$. Performing Reidemeister I and II moves on the red produces the simplified diagram in frame $9$. Frame $10$ follows from Reidemeister III moves, then Reidemeister II moves along the blue cusps. The transition to frame $11$ is identical to the maneuver described in frames $3,4,5$: perform a Reidemeister I at the top of the black knot, pass the resulting crossing through the red and blue strands via Reidemeister III moves, then untwist the resulting Reidemeister I at the bottom of the black knot. The final frame $12$ then follows from Reidemeister II moves which poke the blue zig-zags across the black knot. 
\end{proof}

\begin{figure}[ht]
	\centering
	\begin{overpic}[width=\textwidth]{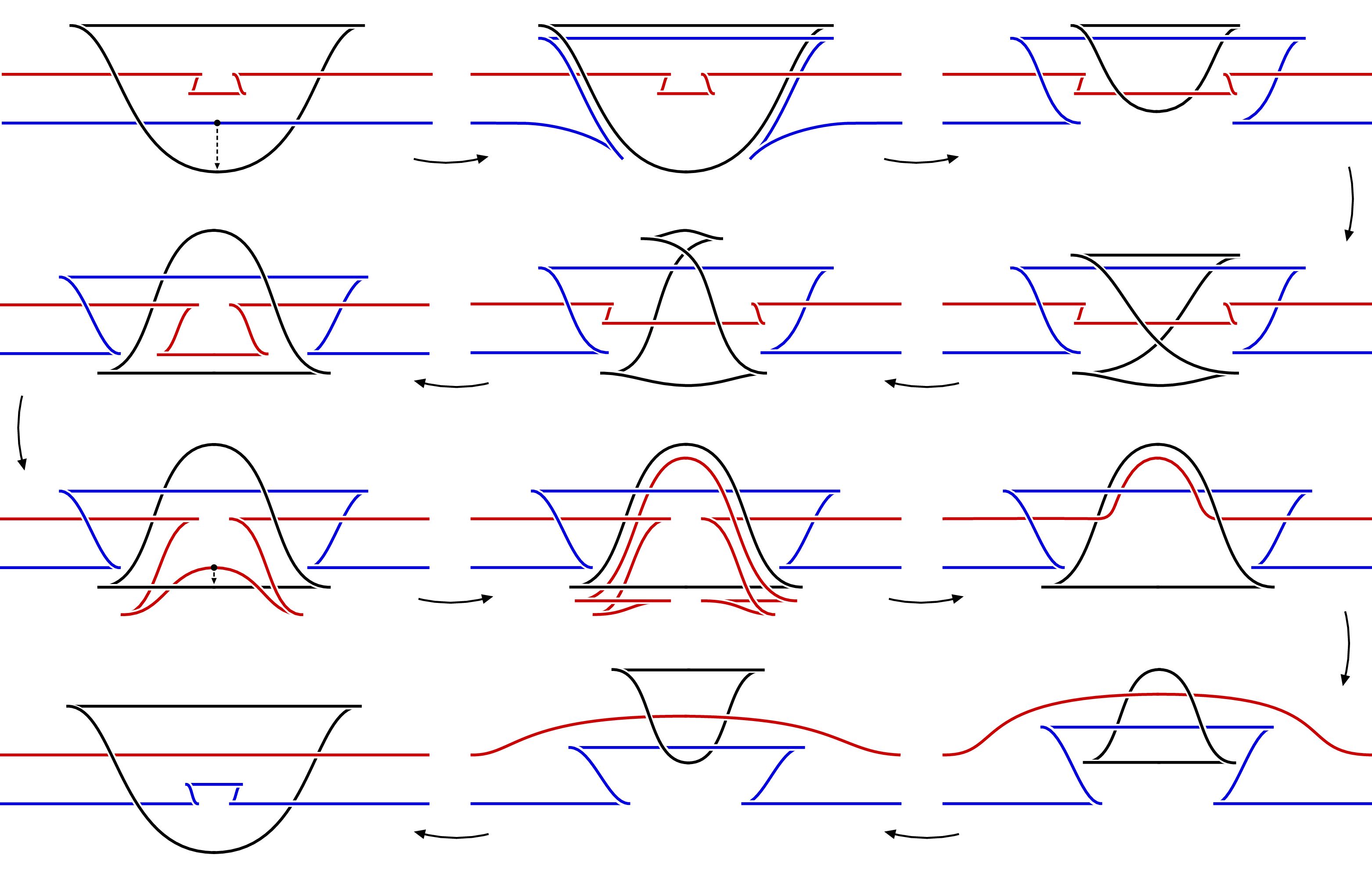}

        \put(14.25,13){\tiny $(+1)$}
        \put(14,32){\tiny $(+1)$}
        \put(14,47.5){\tiny $(+1)$}
        \put(14.25,62.5){\tiny $(+1)$}

        \put(48.5,15.25){\tiny $(+1)$}
        \put(48.5,32){\tiny $(+1)$}
        \put(48.25,47.5){\tiny $(+1)$}
        \put(48.5,62.5){\tiny $(+1)$}

        \put(83,15.25){\tiny $(+1)$}
        \put(82.75,32){\tiny $(+1)$}
        \put(82.75,45.75){\tiny $(+1)$}
        \put(82.5,62.5){\tiny $(+1)$}
       
	\end{overpic}
	\caption{The proof of \cref{lemma:main}.}
	\label{fig:main-lemma-proof-new}
\end{figure}

With \cref{lemma:main} in hand, we construct a family of Lagrangian slice knots following \cite{conway2021symplectic}. Fix $N\geq 1$. Consider the Weinstein handlebody diagram in \cref{fig:main-handle-diagram}, where two Weinstein $1$-handles correspond to the contact-$(+1)$ surgeries on the dotted max-tb unknots, and two Weinstein $2$-handles are attached along $\Lambda_R, \Lambda_B$, inducing contact-$(-1)$ surgeries.

\begin{figure}[ht]
	\centering
    \begin{overpic}[width=0.9\textwidth]{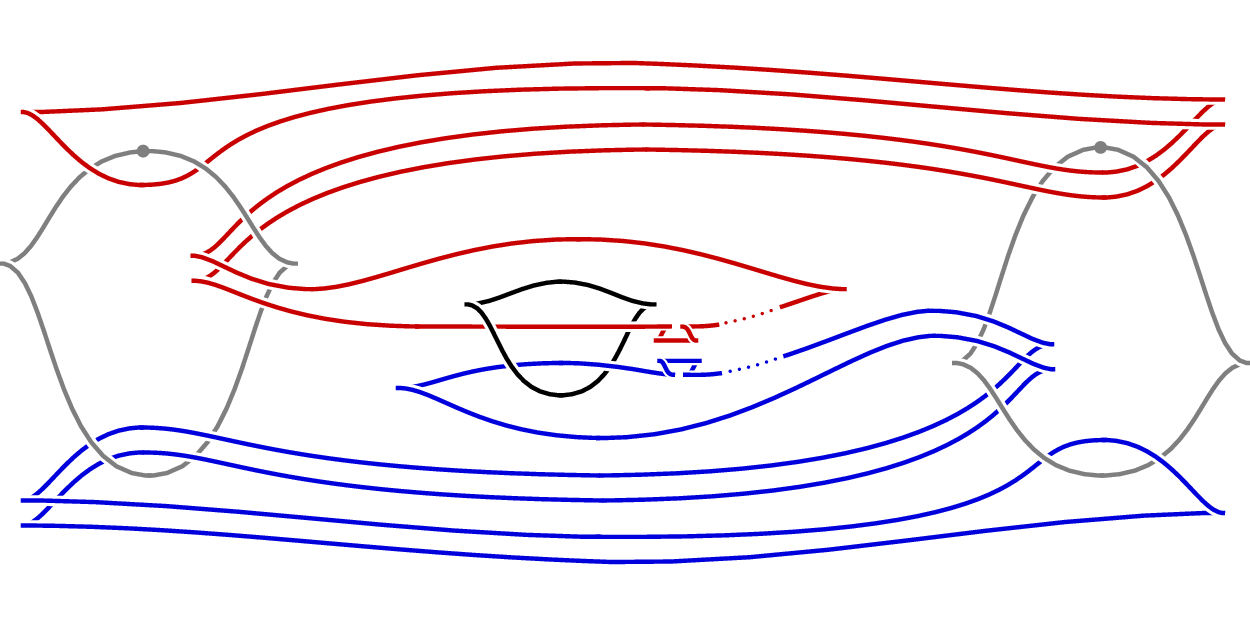}
       \put(43,28.5){\footnotesize $\Lambda_{n,N}$}

       \put(20,16){\footnotesize \textcolor{darkblue}{$\Lambda_B$}}
       \put(26,15){\footnotesize \textcolor{darkblue}{$(-1)$}}

       \put(67,34){\footnotesize \textcolor{darkred}{$\Lambda_R$}}
       \put(73,33){\footnotesize \textcolor{darkred}{$(-1)$}}

       \put(-2,20){\footnotesize \textcolor{gray}{$(+1)$}}

       \put(98,28){\footnotesize \textcolor{gray}{$(+1)$}}

       \put(70,4){\tiny \textcolor{darkblue}{$n$ double $\pm$-stabilizations}}
       \put(71,45.25){\tiny \textcolor{darkred}{$2N-n$ double $\pm$-stabilizations}}
       
	\end{overpic}
	\caption{A handlebody presentation of the knot $\Lambda_{n,N}$.}
	\label{fig:main-handle-diagram}
\end{figure}

Ignoring the knot $\Lambda_{n,N}$, the attaching spheres $\Lambda_R, \Lambda_B$ may be Legendrian isotoped so that each dotted $1$-handle is a max-tb meridian around each attaching sphere. This implies that, up to Weinstein homotopy, the Weinstein handles may be put in canceling position and thus the handlebody diagram is a presentation of standard $B^4$.

\begin{proof}[Proof of \eqref{part:main-construction} of \cref{thm:main}.]
By \cite[Theorem 1.10]{conway2021symplectic}, $\Lambda_{n,N}$ bounds a regular Lagrangian disk $D_{n,N}$ in $B^4$. Indeed, before attaching the Weinstein handles, $\Lambda_{n,N}$ is a max-tb unknot and hence bounds a standard Lagrangian disk in $B^4$. The disk $D_{n,N}$ is the image of this standard disk in $B^4$ after attaching the Weinstein handles.

The exterior $X_{D_{n,N}}$ is obtained by converting $\Lambda_{n,N}$ into a dotted contact-$(+1)$ surgery. Then, by repeatedly applying \cref{lemma:main} to the resulting handle decompositions of the exteriors, we may exchange double stabilizations between $\Lambda_R$ and $\Lambda_B$ to obtain Weinstein handlebody equivalences between $X_{D_{n,N}}$ and $X_{D_{m,N}}$ for all $n\neq m \in \{1,\dots, 2N\}$. This establishes that the knots $\Lambda_{n,N}$ as presented in the boundary of the handlebody diagram in \cref{fig:main-handle-diagram} satisfy \eqref{part:main-construction} of \cref{thm:main}. To complete the proof, it remains to establish that the knot $\Lambda_{n,N}$ in \cref{fig:main-handle-diagram} becomes the knot $\Lambda_{n,N}$ in \cref{fig:main} after canceling the Weinstein handles.

\begin{figure}[ht]
	\centering
    \begin{overpic}[width=0.8\textwidth]{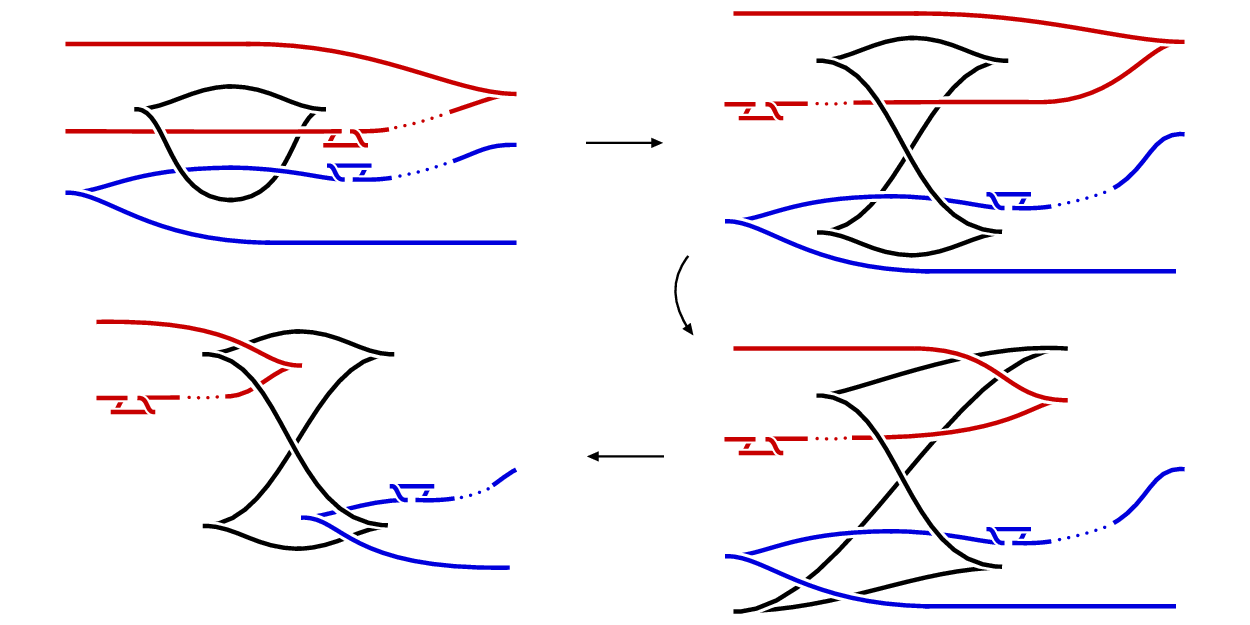}
       
	\end{overpic}
	\caption{A sequence of Legendrian Reidemeister moves near $\Lambda_{n,N}$.}
	\label{fig:main-handle-diagram-proof1}
\end{figure}

The first step of the necessary handle calculus is the local Legendrian isotopy depicted in \cref{fig:main-handle-diagram-proof1}. Once the attaching spheres are appropriately interlocked with cusps of $\Lambda_{n,N}$, a straightforward sequence of Legendrian isotopies retracting $\Lambda_R$ and $\Lambda_B$ yields the equivalent surgery diagram in \cref{fig:main-handle-diagram-proof2}.

\begin{figure}[ht]
	\centering
    \begin{overpic}[width=0.9\textwidth]{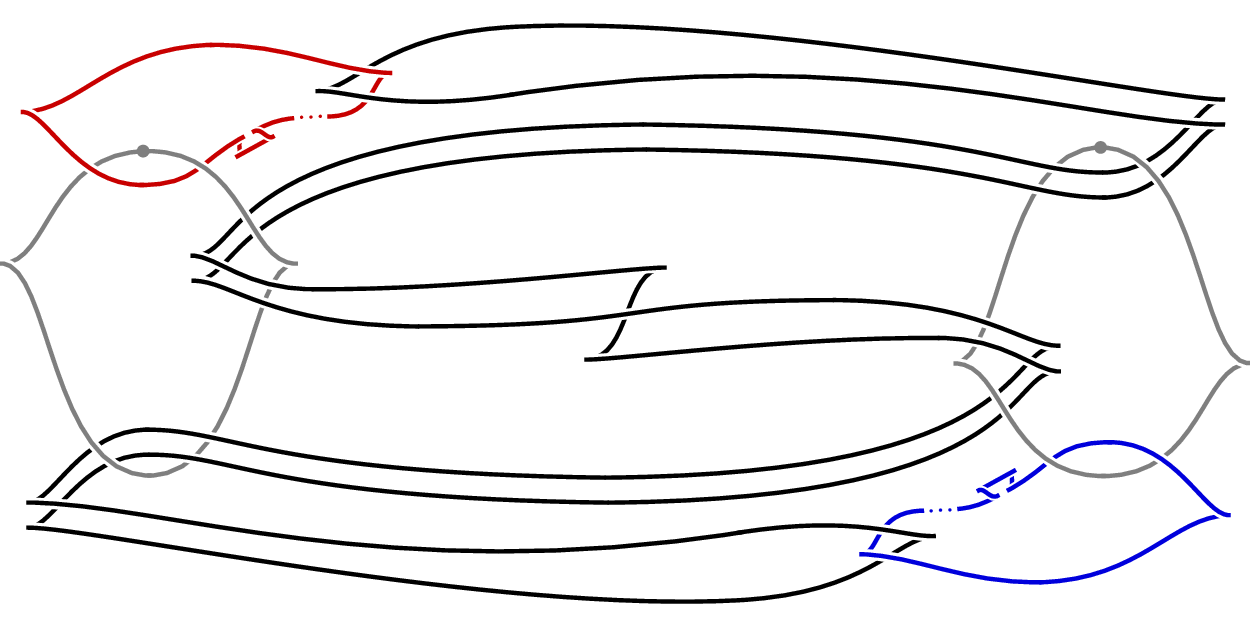}
       \put(46,30.5){\small $\Lambda_{n,N}$}

       \put(98,12){\footnotesize \textcolor{darkblue}{$(-1)$}}

       \put(-2,36){\footnotesize \textcolor{darkred}{$(-1)$}}

       \put(-2,20){\footnotesize \textcolor{gray}{$(+1)$}}

       \put(98,28){\footnotesize \textcolor{gray}{$(+1)$}}

       \put(73,1){\tiny \textcolor{darkblue}{$n$ double $\pm$-stabilizations}}
       \put(4,47.75){\tiny \textcolor{darkred}{$2N-n$ double $\pm$-stabilizations}}
       
	\end{overpic}
	\caption{An alternative presentation of $\Lambda_{n,N}$.}
	\label{fig:main-handle-diagram-proof2}
\end{figure}

Next, we perform a sequence of handleslides of $\Lambda_{n,N}$ across the contact-$(-1)$ surgeries to disconnect the former from the $1$-handles. The initial isotopies and slides across $\Lambda_B$ are depicted locally in \cref{fig:main-handle-diagram-proof3}. From the first to the second panel we perform two slides indicated by the arrow, and the third panel is obtained from additional Legendrian Reidemeister moves meant to prepare for the next round of handleslides. We have also introduced the box notation for ease of presentation.

\begin{figure}[ht]
	\centering
    \begin{overpic}[width=\textwidth]{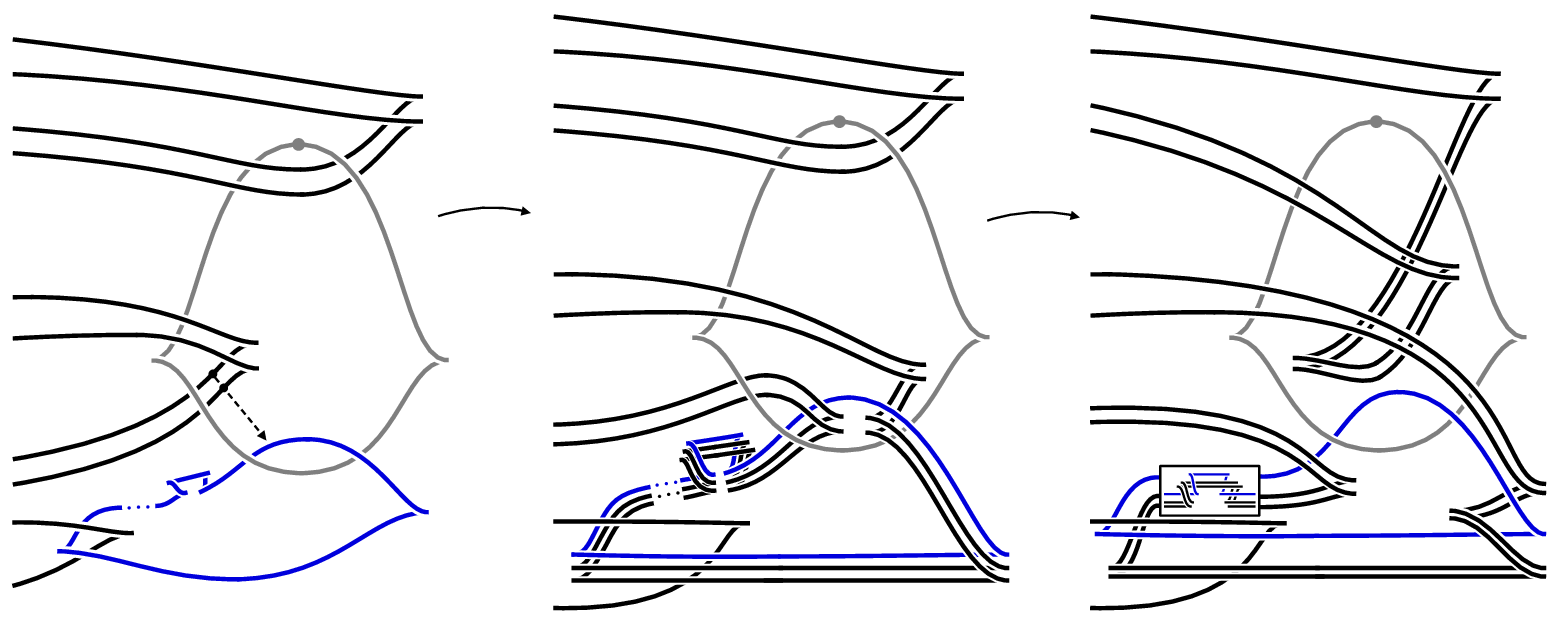}

       \put(87.5,7){\footnotesize \textcolor{darkblue}{$(-1)$}}
       \put(51.75,5.5){\footnotesize \textcolor{darkblue}{$(-1)$}}

       \put(23,3){\footnotesize \textcolor{darkblue}{$(-1)$}}

       \put(8.5,25){\footnotesize \textcolor{gray}{$(+1)$}}

       \put(95,25){\footnotesize \textcolor{gray}{$(+1)$}}

       \put(43,25){\footnotesize \textcolor{gray}{$(+1)$}}

	\end{overpic}
	\caption{Two contact-$(-1)$ handleslides, then Reidemeister moves.}
	\label{fig:main-handle-diagram-proof3}
\end{figure}

\cref{fig:main-handle-diagram-proof4} then performs two more handleslides of $\Lambda_{n,N}$ over $\Lambda_B$. After a sequence of Reidemeister moves, we obtain the result in the third panel, where the dotted $1$-handle is a meridional contact-$(+1)$ surgery for the contact-$(-1)$ surgery along $\Lambda_B$.

\begin{figure}[ht]
	\centering
    \begin{overpic}[width=\textwidth]{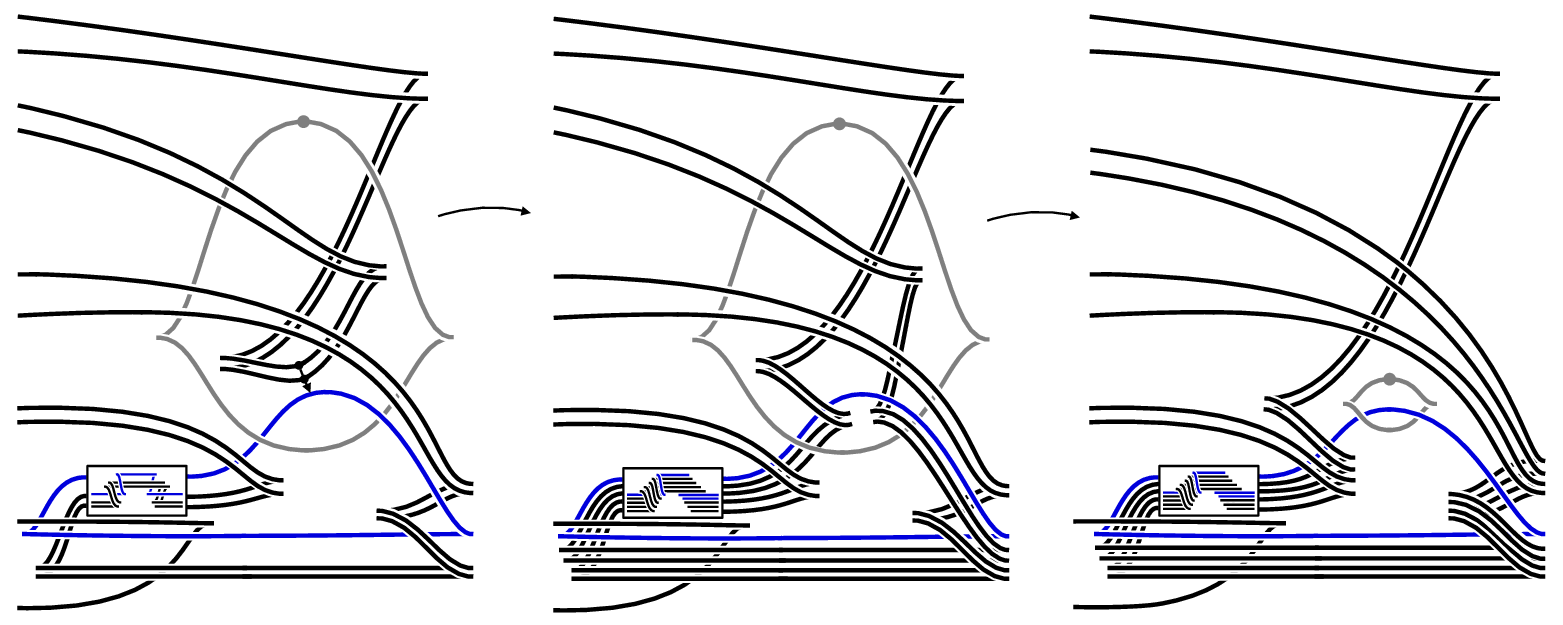}

       \put(87.5,7){\footnotesize \textcolor{darkblue}{$(-1)$}}
       \put(53.5,7){\footnotesize \textcolor{darkblue}{$(-1)$}}

       \put(19,7){\footnotesize \textcolor{darkblue}{$(-1)$}}

       \put(8.5,25){\footnotesize \textcolor{gray}{$(+1)$}}

       \put(87.5,10){\footnotesize \textcolor{gray}{$(+1)$}}

       \put(43,25){\footnotesize \textcolor{gray}{$(+1)$}}

	\end{overpic}
	\caption{Two more contact-$(-1)$ handleslides and Reidemeister moves.}
	\label{fig:main-handle-diagram-proof4}
\end{figure}

The same sequence (rotated by $\pi$ radians) provides the necessary steps near $\Lambda_R$. Afterwards, both $\Lambda_R$ and $\Lambda_B$ are in geometrically canceling position with $1$-handles and may be erased via a Weinstein homotopy \cite{ding2009handle}. Left in the wake of the cancellation is the knot $\Lambda_{n,N}$ as depicted in \cref{fig:main}.
\end{proof}

\section{Obstruction}\label{sec:obstruction}

In this section we prove \eqref{part:main-obstruction} of \cref{thm:main} by showing that the topological knot types $K_{n,N}$ of the Legendrian knots $\Lambda_{n,N}$, $1\leq n \leq N$, are distinct for sufficiently large $N$. For convenience, we have redrawn and simplified the underlying topological knot diagram in \cref{fig:main-smooth} by collecting full twists into the boxes.

\begin{figure}[ht]
	\centering
    \begin{overpic}[width=0.9\textwidth]{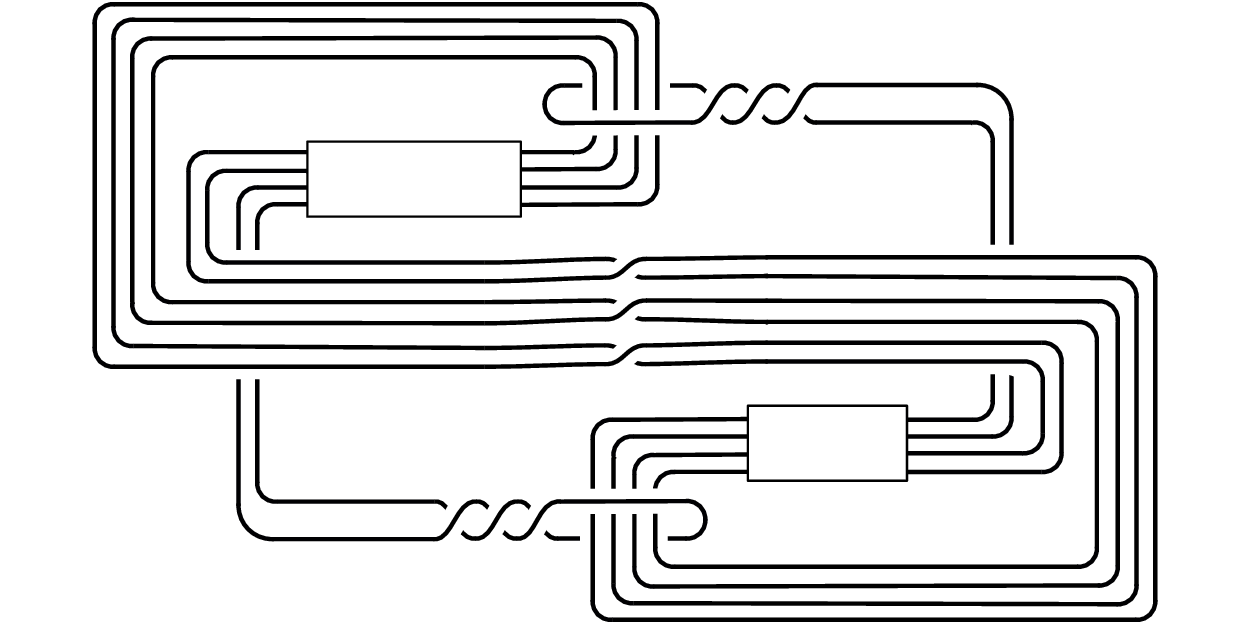}
       \put(10.5,13){$K_{n,N}$}
       \put(61.25,13.75){\small $-2n -2$}
       \put(25.25,34.75){\footnotesize $-2(2N - n)-2$}
	\end{overpic}
	\caption{The topological knots $K_{n,N}$ in the statement of \cref{thm:main}. Negative integral boxes indicate full left-handed twists.}
	\label{fig:main-smooth}
\end{figure}

A further sequence of smooth isotopies to simplify the knot diagram is presented in \cref{fig:main-smooth-iso}. The transition from the first panel to the second panel is obtained by pulling the central three bands over the $-2n-2$ twist box. This induces a full right-handed twist which we place in a box near the left side of the diagram.

The passage to the lower-right diagram involves three steps. First, the full twist on the two rightmost bands passing through the $1$-box is slid down and to the right to be absorbed into the rightmost twist box. Second, the leftmost band which had passed through the $1$-box now wraps once around the other two bands, and has a full positive twist along the band itself; we may cancel this full band twist by zipping two of the negative half-twists along the band. Third, we zip the three negative half-twists in the upper right corner near the ribbon singularity along the band and collect them with the three negative half-twists located near the other ribbon singularity into a twist box with $-3$ full twists.

\begin{figure}[ht]
	\centering
    \begin{overpic}[width=\textwidth]{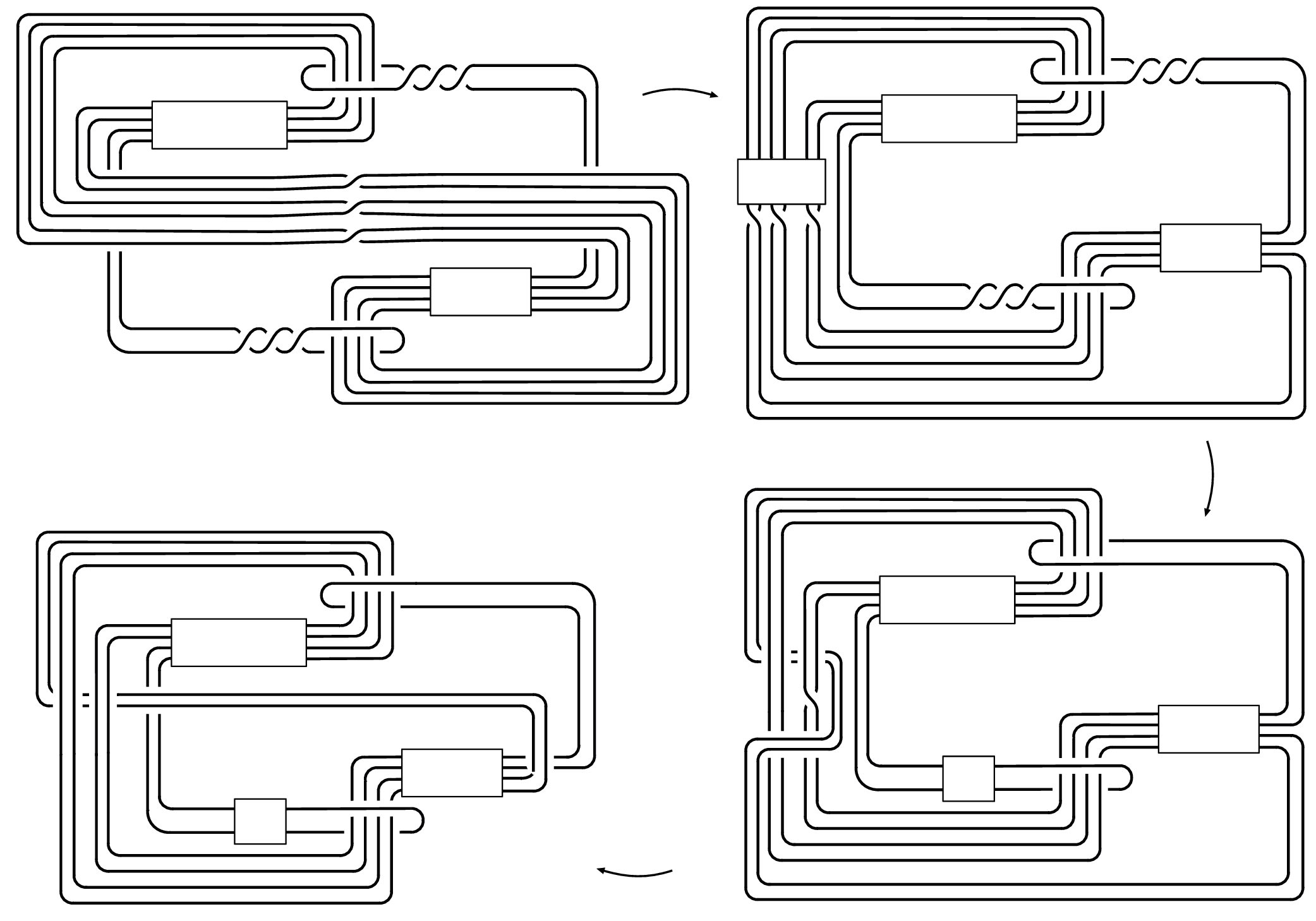}
      
       \put(33.25,47.25){\tiny \textcolor{black}{$-2n-2$}}
       \put(11.5,59.85){\tiny \textcolor{black}{$-4N + 2n-2$}}

       \put(58.85,55.75){\tiny $1$}
       \put(88.75,50.75){\tiny \textcolor{black}{$-2n-2$}}
       \put(67, 60.5){\tiny \textcolor{black}{$-4N + 2n-2$}}

       \put(72.35,10.4){\tiny $-3$}
       \put(88.75,14){\tiny \textcolor{black}{$-2n-1$}}
       \put(66.8,23.8){\tiny \textcolor{black}{$-4N + 2n-2$}}

       \put(18.4,7){\tiny $-4$}
       \put(31.3,10.75){\tiny \textcolor{black}{$-2n-1$}}
       \put(13,20.75){\tiny \textcolor{black}{$-4N + 2n-2$}}
	\end{overpic}
	\caption{Simplifying $K_{n,N}$.}
	\label{fig:main-smooth-iso}
\end{figure}

To obtain the lower-left panel of \cref{fig:main-smooth-iso}, we then take the negative half-twist near the left side of the diagram and zip it down through the $(-2n-1)$-box and up toward the upper-right ribbon singularity. We then use the half-twist to swap the over and under strands along the ribbon singularity. Finally, we flip the lowermost band of the diagram across the $(-3)$-box and the $(-2n-1)$-box, inducing a full negative twist into the band, which we zip along and absorb into the $(-3)$-box.

\begin{proof}[Proof of \eqref{part:main-obstruction} of \cref{thm:main}.]
A further redrawing of $K_{n,N}$ is given in \cref{fig:main-hyperbolic}. Via Rolfsen twists, $K_{n,N}$ is the image of the knot $K$ on the right side of the figure under Dehn surgery on the unknots $R$ and $B$ with slopes $\frac{1}{4N - 2n + 2}$ and $\frac{1}{2n+1}$, respectively.  

\begin{figure}[ht]
	\centering
    \begin{overpic}[width=0.85\textwidth]{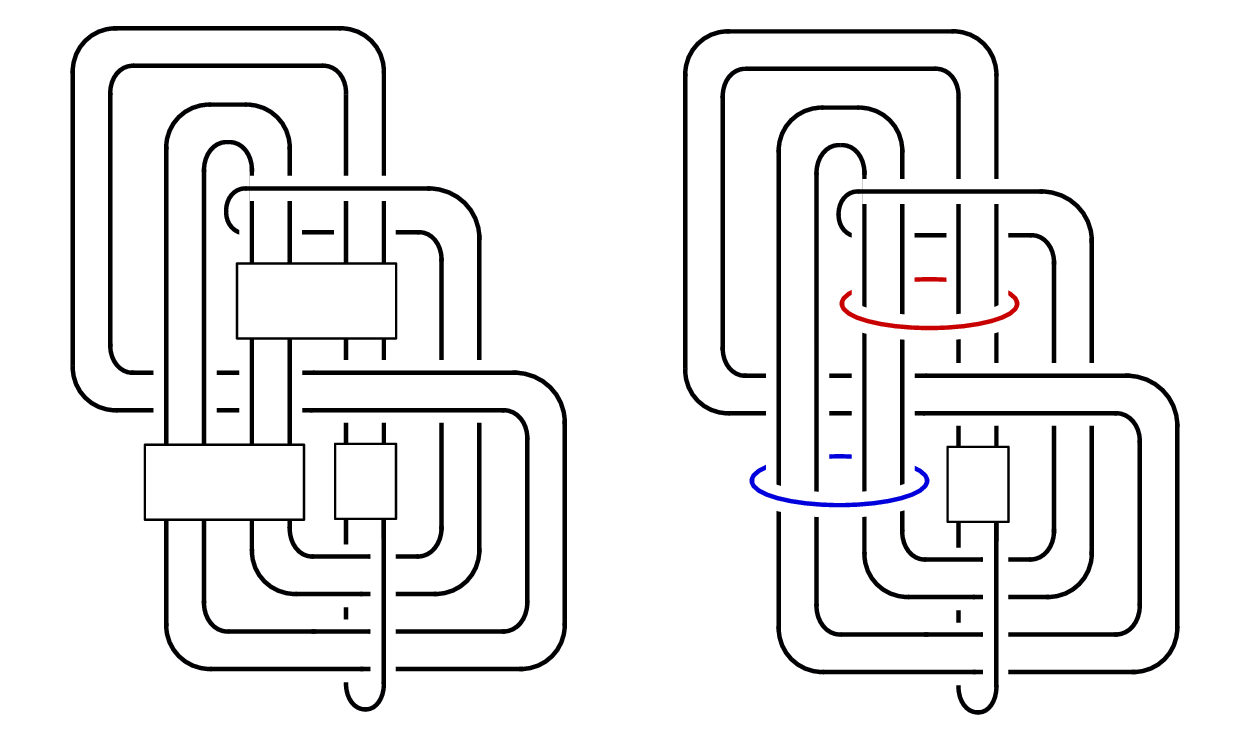}
       \put(5.5,10){$K_{n,N}$}
       \put(13.25,19){\footnotesize $-2n-1$}
       \put(27.75,19){\footnotesize $-4$}
       \put(19.25,33.5){\tiny $-4N+2n-2$}

       \put(55,19.25){\footnotesize \textcolor{darkblue}{$\frac{1}{2n+1}$}}
       \put(89,33.25){\footnotesize \textcolor{darkred}{$\frac{1}{4N - 2n + 2}$}}
       \put(76.75,19){\footnotesize $-4$}

       \put(58,10){$K$}
       \put(66.5,15){\tiny \textcolor{darkblue}{$B$}}
       \put(73.6,29.5){\tiny \textcolor{darkred}{$R$}}

	\end{overpic}
	\caption{The hyperbolic link $K\cup R \cup B$.}
	\label{fig:main-hyperbolic}
\end{figure}

The link $K \cup R \cup B$ has DT code
\begin{multline*}
[(-28,-56,-92,-90,-88,-86,-74,-98,46,-32,-60,14,44,-96,-106,78,\\
-50,102,22,110,4,-72,-80,-34,64,18,-26,-94,-108,76,-104,48,20,112,\\
2,54,84,-38,-66,-52,100,-24,-12,-10,-8,-6,42,70),\\
(36,82,-16,-62),(-30,68,40,-58)].
\end{multline*}
Using the \texttt{verify\_hyperbolicity()} command in SnapPy with Sage \cite{snappy}, we confirm that $K \cup R \cup B$ is hyperbolic with volume $\approx 23.449$.

Thurston \cite{thurston1980geometry} showed that a hyperbolic manifold with $k$ cusps remains hyperbolic after Dehn filling with slopes $\tfrac{p_1}{q_1}, \dots, \tfrac{p_k}{q_k}$ provided $\sum_{i=1}^k p_i^2 + q_i^2$ is sufficiently large. Consequently, provided $N$ is sufficiently large, the complement of $K_{n,N}$ (which is obtained from $M:= S^3 \setminus (K \cup R \cup B)$ by Dehn filling $R \cup B$ with slopes $\tfrac{p_R}{q_R} = \tfrac{1}{4N - 2n + 2}$ and $\tfrac{p_B}{q_B} = \tfrac{1}{2n+1}$) is hyperbolic for all $n$; see \cref{fig:main-slopes}.

Neumann and Zagier \cite{neumann1985volumes} made this more precise by computing volumes of Dehn fillings to high-order precision. As the squared length of a $(1, \ell)$-curve on a square torus of unit area is $1 + \ell^2$, their formula gives 
\begin{equation}\label{eq:volumes}
\mathrm{Vol}(S^3 \setminus K_{n,N}) = \mathrm{Vol}(M) - \pi^2\left(\frac{1}{1 + (2n+1)^2} + \frac{1}{1 + (4N - 2n + 2)^2}\right)
+ \mathcal{O}\left(\frac{1}{|(\vec{p},\vec{q})|^4}\right).    
\end{equation}
To distinguish our knots, we then appeal to the following elementary lemma.

\begin{lemma}
For $C \gg 0$, $f(x) = \frac{1}{1+x^2} + \frac{1}{1 + (C - x)^2}$ is strictly monotonic on the interval $(\tfrac{1}{\sqrt{3}},\tfrac{C}{2})$.  
\end{lemma}

\begin{proof}
Differentiating, we find that critical points of $f$ occur when 
\begin{equation}\label{eq:derivative}
\frac{x}{(1+x^2)^2} = \frac{C-x}{(1+(C-x)^2)^2}. 
\end{equation}
Note that $x = \tfrac{C}{2}$ is a critical point. We claim there are no critical points on the interval $(\tfrac{1}{\sqrt{3}},\tfrac{C}{2})$. To see this, note that 
\[
\frac{d}{d\tau} \, \frac{\tau}{(1+\tau^2)^2} = \frac{1-3\tau^2}{(1+\tau^2)^3} < 0 \quad \text{ for } \quad |\tau| > \frac{1}{\sqrt{3}}. 
\]
In particular, the left side of \eqref{eq:derivative} is monotonically decreasing on the interval $(\tfrac{1}{\sqrt{3}},\tfrac{C}{2})$, while the right side is monotonically increasing. As the left and right sides are equal when $x = \tfrac{C}{2}$, there cannot be a solution to \eqref{eq:derivative} on $(\tfrac{1}{\sqrt{3}},\tfrac{C}{2})$. The monotonicity claim follows.  
\end{proof}

The lemma implies that the numbers $\tfrac{1}{1 + (2n+1)^2} + \tfrac{1}{1 + (4N - 2n + 2)^2}$ are pairwise distinct for all $2n+1\in (\tfrac{1}{\sqrt{3}},\tfrac{4N+3}{2})$. In particular, they are pairwise distinct for $n\in \{1, \dots, N\}$. The volume formula \eqref{eq:volumes} then implies that the knots $K_{1,N}, \dots, K_{N,N}$ are smoothly pairwise distinct, completing the proof of the theorem.
\end{proof}

\begin{figure}[ht]
	\centering
    \begin{overpic}[width=0.75\textwidth]{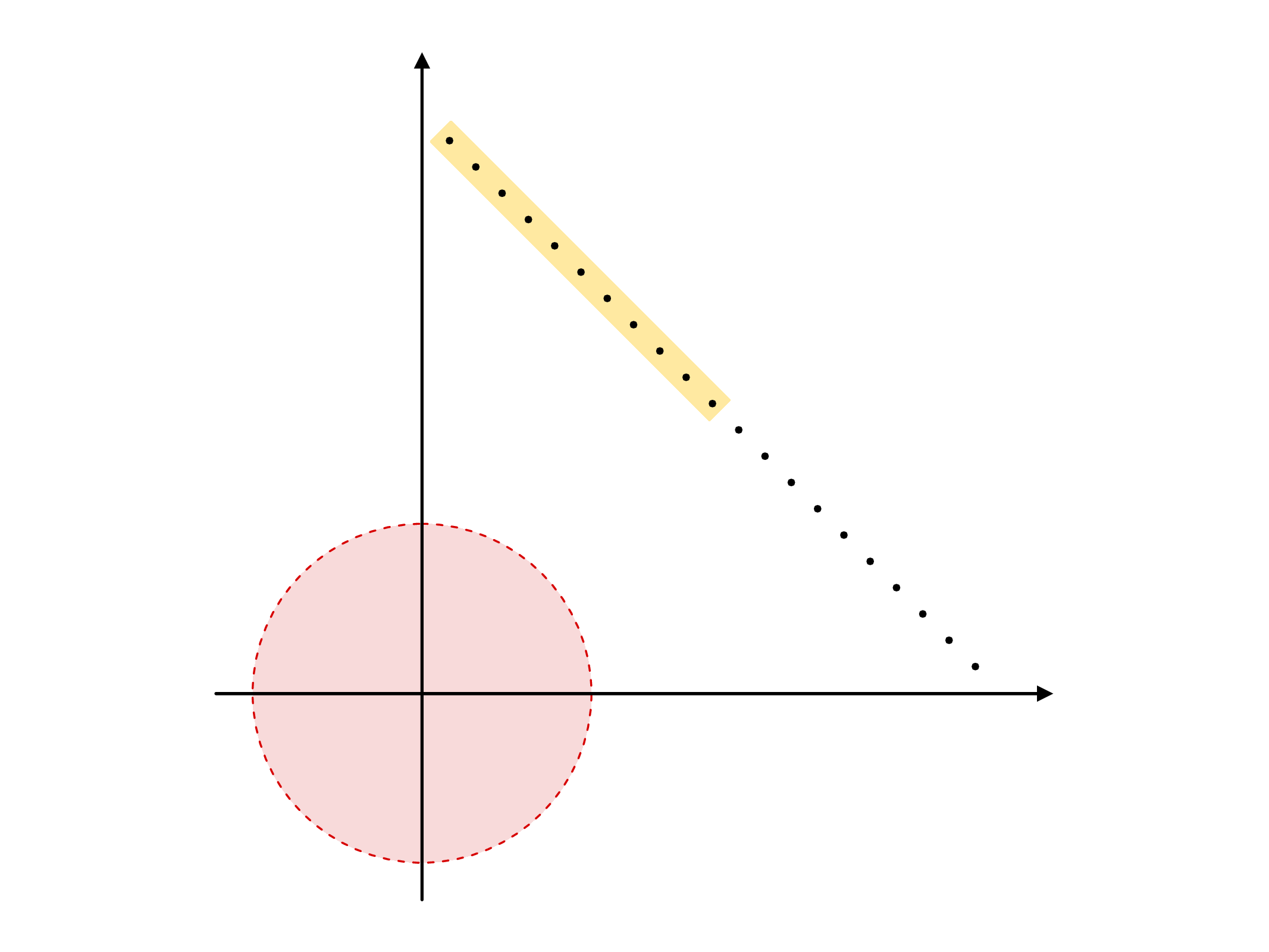}
       \put(84,19.5){\small $q_B$}
       \put(31.75,73.5){\small $q_R$}

       \put(65.75,40.75){\small $q_B + q_R = 4N + 3$}
	\end{overpic}
	\caption{The projection of $(p_B, q_B, p_R, q_R)$-space to the $(q_B,q_R)$ plane. The red disk represents the region of short slopes where Dehn fillings are potentially not hyperbolic. The dotted line represents the slopes $p_B = p_R = 1$ and $q_B + q_R = 4N + 3$. The highlighted region indicates the volumes we distinguish in the proof of \eqref{part:main-obstruction} of \cref{thm:main}.}
	\label{fig:main-slopes}
\end{figure}

\section{Connected summation}\label{sec:connected-summation}

In this section we prove \cref{thm:KplusKbar}. In lieu of a generalization of the $\mathrm{TB}$ argument in \cref{remarK:comparison}, one may similarly hope to obstruct Lagrangian sliceness with the maximal self-linking number $\mathrm{SL}$, an invariant which is also $-1$ for such knots; indeed, a Lagrangian slice knot is algebraic by combined work of \cite{rudolph1983seifertribbons,eliashberg1995pushoff,boileau2001quasipositive} and thus maximizes the slice-Bennequin inequality $\mathrm{SL} \leq -1$ \cite{rudolph1993quasipositivity}. Regarding connected sums, the corresponding formula for maximal self-linking number appears to have been known to experts but has not been recorded in the literature. We thank John Etnyre for suggesting the proof. 

\begin{lemma}\label{lemma:sum}
For all topological knot types $K_1,K_2$ we have $\mathrm{SL}(K_1\, \# \, K_2) = \mathrm{SL}(K_1)+ \mathrm{SL}(K_2) + 1$.
\end{lemma}

\begin{proof}
The inequality $\mathrm{SL}(K_1\, \# \, K_2) \geq \mathrm{SL}(K_1)+ \mathrm{SL}(K_2) + 1$ follows from the elementary calculation that if $T_1, T_2$ are transverse knots, then $\mathrm{sl}(T_1 \, \# \, T_2) = \mathrm{sl}(T_1) + \mathrm{sl}(T_2) + 1$. Thus, we only need to prove $\mathrm{SL}(K_1\, \# \, K_2) \leq \mathrm{SL}(K_1)+ \mathrm{SL}(K_2) + 1$.

Let $T$ be a transverse knot in the topological type $K_1\, \# \, K_2$. It is known \cite[\S 2.9]{etnyre2005surveyknots} that there is a Legendrian knot $\Lambda$ whose standard positive transverse pushoff $T_+(\Lambda)$ is transversely isotopic to $T$. Moreover, $\mathrm{sl}(T) = \mathrm{sl}(T_+(\Lambda)) = \mathrm{tb}(\Lambda) + \mathrm{rot}(\Lambda)$. By Etnyre and Honda's classification of Legendrian connected sums \cite[Theorem 3.4]{etnyre2003connected}, $\Lambda$ is Legendrian isotopic to $\Lambda_{1} \,\#\, \Lambda_2$, where $\Lambda_i$ is a Legendrian representative of $K_i$ for $i\in \{1,2\}$. Note that $T_+(\Lambda_i)$ is a transverse representative of $K_i$ with $\mathrm{sl}(T_+(\Lambda_i)) = \mathrm{tb}(\Lambda_i) + \mathrm{rot}(\Lambda_i)$. Applying the $\mathrm{tb}$ and $\mathrm{rot}$ connected sum formulas of \cite[Lemma 3.3]{etnyre2003connected}, we have
\begin{align*}
    \mathrm{sl}(T) &= \mathrm{tb}(\Lambda_{1} \,\#\, \Lambda_2) + \mathrm{rot}(\Lambda_{1} \,\#\, \Lambda_2) \\
    &= \left(\mathrm{tb}(\Lambda_1) + \mathrm{tb}(\Lambda_2) + 1\right)\, + \, \left(\mathrm{rot}(\Lambda_1) + \mathrm{rot}(\Lambda_2)\right) \\
    &= \left(\mathrm{tb}(\Lambda_1) + \mathrm{rot}(\Lambda_1)\right) \, + \, \left(\mathrm{tb}(\Lambda_2) + \mathrm{rot}(\Lambda_2)\right) + 1 \\
    &= \mathrm{sl}(T_+(\Lambda_1)) + \mathrm{sl}(T_+(\Lambda_2)) + 1 \\
    &\leq \mathrm{SL}(K_1)+ \mathrm{SL}(K_2) + 1.
\end{align*}
Maximizing over all transverse representatives $T$ gives the desired result. 
\end{proof}

We denote the HOMFLYPT polynomial \cite{freyd1985new,przytycki1987invariants} of an oriented link by $P_L(v,z)$, so that it is defined by $P_{\mathrm{unknot}}(v,z) = 1$ and the diagrammatic skein relation
\[
v^{-1}\, P_{L_+}(v,z) - v\, P_{L_-}(v,z) = z\, P_{L_0}(v,z),
\]
where $L_0$ is obtained from $L$ by performing an orientation-preserving resolution of a single crossing, and $L_{\pm}$ replaces the crossing with a positive (resp.\ negative) crossing. The next lemma also appears to be known to experts (for example, see \cite[\S 3]{acampo2024two}) but we could not locate written details. We supply a proof for completeness. 

\begin{lemma}\label{lemma:homfly-lemma}
If $P_K(v,z)$ denotes the HOMFLYPT polynomial, then $P_K(v, v^{-1}-v) =1$.     
\end{lemma}

\begin{proof}
Let $f_K(v):=P_K(v, v^{-1}-v)$. The defining skein relation $v^{-1}\, P_{L_+}(v,z) - v\, P_{L_-}(v,z) = z\, P_{L_0}(v,z)$ gives the skein relation 
\begin{equation}\label{eq:skein-specialization}
v^{-1}\, f_{L_+}(v) - v\, f_{L_-}(v) = (v^{-1}-v)\, f_{L_0}(v).    
\end{equation}
We will show $f_K(v) = 1$ by induction on the number $n$ of crossings in a diagram. 

For the base case $n=0$ we consider the unlink. As $P_{L \sqcup L'}(v,z) = \tfrac{v^{-1}-v}{z}\, P_{L \#L'}(v,z)$, it follows that $f_{L\sqcup L'}(v) = f_{L\# L'}(v)$ and consequently $f_{\mathrm{unlink}}(v) = 1$. 

Now assume inductively that if $D_{n-1}$ is a diagram with $n-1$ crossings then $f_{D_{n-1}}(v) = 1$. Let $D_n^{\pm}$ be diagrams with $n$ crossings which differ only at one $\pm$-crossing. The skein relation \eqref{eq:skein-specialization} and the inductive hypothesis then give
\[
v^{-1}\, f_{D_n^+}(v) - v\, f_{D_n^-}(v) = v^{-1} - v.
\]
Rearranging, we see that 
\[
v^{-1}\cdot(f_{D_n^+}(v) - 1) = v\cdot(f_{D_n^-}(v) - 1).
\]
In other words, changing a crossing in a diagram $D_{n}$ with $n$ crossings modifies the quantity $f_{D_n}(v) - 1$ by a factor of $v^{\pm 2}$, depending on the sign of the change. Starting with $D_n$, there is a sequence of $k$ crossing changes which takes $D_n$ to the unlink. Thus, for some $\ve_i\in \{\pm 1\}$, we have  
\begin{align*}
    f_{D_n}(v) - 1 &= v^{2\ve_1}\cdot v^{2\ve_2}\cdots v^{2\ve_k} \cdot (f_{\mathrm{unlink}}(v) - 1) = 0,
\end{align*}
hence $f_{D_n}(v) = 1$. This completes the proof. 
\end{proof}

\begin{proof}[Proof of \cref{thm:KplusKbar}.]
A consequence of the MFW inequality \cite{morton1986seifert,franks1987braids} --- see the discussion after \cite[Question 1.5]{ng2012arc} --- is the upper bound 
\begin{equation*}
\mathrm{SL}(K)+\mathrm{SL}(\overline{K}) \leq -\mathrm{span}_v\, P_K(v,z) - 2,
\end{equation*}
where $\mathrm{span}_v$ is the difference between the maximum and minimum $v$-degrees. Following \cref{lemma:sum},
\begin{equation*}
    \mathrm{SL}(K\, \# \, \overline{K}) \, =\,  \mathrm{SL}(K) + \mathrm{SL}(\overline{K}) + 1 \, \leq\,  -\mathrm{span}_v\, P_K(v,z) - 1.
\end{equation*}
Lagrangian sliceness of $K\, \# \, \overline{K}$ then requires $\mathrm{SL}(K\, \# \, \overline{K}) = -1$, which by the above estimate forces $\mathrm{span}_v\, P_K(v,z) = 0$. In other words, we have $P_K(v,z) = v^m\, g(z)$ for some integer $m$ and a Laurent polynomial $g(z)$. \cref{lemma:homfly-lemma} then implies $g(v^{-1}- v) = v^{-m}$. For a fixed $z\in \C$, the solutions of $v^{-1} - v = z$ are $v^*_{\pm} = \tfrac{1}{2}(z \pm \sqrt{z^2 + 4})$. This means that for all $z\in \C$, we have 
\[
(v_+^*)^{-m} = g\left((v_{\pm}^*)^{-1} - v_{\pm}^*\right) = (v_-^*)^{-m}.
\]
This is only possible if $m=0$, in which case $g(v^{-1}-v) = 1$. This forces $g(z) = 1$, so $K$ has trivial HOMFLYPT polynomial.
\end{proof}

\bibliography{references}
\bibliographystyle{amsalpha}

\end{document}